\documentclass[12pt]{amsart}
\usepackage{delimset, amsrefs, xcolor, enumitem}
\usepackage[colorlinks, citecolor=cyan]{hyperref}
\usepackage{yhmath} % \wideparen
\usepackage{tikz}
\tikzset{
every picture/.style={thin, >=stealth},
club/.style={smooth, ultra thick},%non-isolated limits of zeros.
eg/.style={blue, thick}%zeros of W5 or W6
}

\usepackage[capitalize]{cleveref}
\crefname{ineq}{Ineq.}{Ineqs.}
\creflabelformat{ineq}{~\upshape(#2#1#3)}
\crefname{itm}{}{}
\creflabelformat{itm}{~\upshape(#2#1#3)}
\crefname{rec}{Recurrence}{Recurrences}
\creflabelformat{rec}{~\upshape(#2#1#3)}

\newtheorem{theorem}{Theorem}[section]
\newtheorem{corollary}[theorem]{Corollary}
\newtheorem{lemma}[theorem]{Lemma}

\newtheorem{proposition}[theorem]{Proposition}

\theoremstyle{definition}
\newtheorem{example}[theorem]{Example}

\numberwithin{equation}{section}
\numberwithin{figure}{section}

\DeclareMathOperator\Cl{\mathrm{Cl}}

\allowbreak
\allowdisplaybreaks

\title[Geometry of zero limits of $(1,2)$-polynomials]{Geometry of limits of zeros \\[5pt] 
of polynomial sequences of type $(1,2)$}

\author[D.G.L. Wang]{David G.L. Wang$^\dag$$^\ddag$}
\address{
$^\dag$School of Mathematics and Statistics, Beijing Institute of Technology, 102488 Beijing, P. R. China\\
$^\ddag$Beijing Key Laboratory on MCAACI, Beijing Institute of Technology, 102488 Beijing, P. R. China}
\email{glw@bit.edu.cn}

\author[J.J.R. Zhang]{Jerry J.R. Zhang$^\dag$}
\address{
$^\dag$School of Mathematics and Statistics, Beijing Institute of Technology, 102488 Beijing, P. R. China}
\email{jrzhang.combin@gamil.com}
\thanks{Wang is supported by the General Program of National Natural Science Foundation of China (Grant No.\ 11671037). This paper is finished when the first author is a visiting scholar at MIT}

\begin{document}

\begin{abstract}   
We study the root distribution of some univariate polynomials satisfying a recurrence of order two with linear and quadratic polynomial coefficients. We show that the set of non-isolated limits of zeros of the polynomials is the closure of an arc,  a circle, an interval or its exterior under the real line, the union of at most two different shapes of the above cases expect the union of an arc and a circle, and some degenerate forms.
\end{abstract}

\maketitle
%\tableofcontents

\section{Introduction}

Despite its long history of research, 
the root distribution of polynomials attracts attention.
Many polynomial sequences have not only beautiful root distributions, 
but also applications in analysis, combinatorics, and probability theory,
see Stanley \cite{StaW} and \cite[\S 4]{Sta00} for details.
The limiting distribution of zeros of polynomial has also
been receiving wide attention. 
Beraha, Kahane and Weiss~\cite{BKW75,BKW78} introduced the concept
of limit of zeros for its limiting behaviour. 
They~\cite{BKW80} show the chromatic polynomials 
of two families of maps have limit points of the form $2\brk1{1+\cos(2\pi/n)}$.
Sokal~\cite{Sok04} show that the zeros of chromatic polynomials 
of generalized theta graphs is dense in the whole complex plane 
with a possible exception of the disc $\abs{z-1}<1$. 
Boyer and Goh~\cite{BG07,BG08} show certain ``zero attractors'' of the zero set 
of Euler polynomials and Appell polynomials are highly related to the Szeg\"o curve.
See \cite{GHR09,BM06} for more research on limiting root distribution of polynomials. 

Gross, Mansour, Tucker and the first author~\cite{GMTW16-01,GMTW16-10}
study the root distribution of polynomials satisfying the recurrence
\begin{equation}\label[rec]{rec:AB}
W_n(z) = A(z)W_{n-1}(z)+B(z)W_{n-2}(z),
\end{equation}
where one of the coefficient polynomials $A(z)$ and $B(z)$ 
is linear and the other is a constant.
Wang and Zhang \cite{WZ18X} obtain that 
there are exactly four possible geometric shapes of non-isolated limits 
when both $A(z)$ and $B(z)$ are linear.
Orthogonal polynomials and quasi-orthogonal polynomials have closed relations with
\cref{rec:AB};
see~\cite{ARR99B,BDR04}.
Jin and Wang~\cite{JW17X} characterized common zeros of the polynomials $W_n(z)$ 
for general $A(z)$ and $B(z)$.
The root distribution of the polynomials generated by the function
\[
\sum_{n\ge0}W_n(z)t^n=\frac{1}{1-A(z)t-B(z)t^2}
\]
has been investigated in \cite{Tra14}, 
in which Tran found an algebraic curve containing the zeros 
of all polynomials $W_n(z)$ with large subscript $n$.

This paper is organized as follows. 
Our main results are \cref{thm:spade,thm:club}, 
which are algebraic and geometric descriptions 
of the limiting root distribution of polynomials $W_n(z)$ when $A(z)$ 
is linear and $B(z)$ is quadratic in $z$. 
Note that the zeros of a polynomial are symmetric about the real line, 
so are the limits of zeros. 
It turns out that the set of non-isolated limits of zeros of all $W_n(z)$ 
for a given linear $A(z)$ and a given quadratic $B(z)$, 
subject to some standard assumptions for $W_0(z)$, $W_1(z)$, $A(z)$ and $B(z)$,
has six possible geometric shapes; see \cref{thm:shape:12}.
Some genus polynomials from topological graph theory satisfy \cref{rec:AB} 
and serve as examples; see \cref{ex:ladder,ex:P4}.

\section{The set of limits of zeros}\label{sec:geo}
Let $\{W_n(z)\}_{n\ge0}$ be a polynomial sequence satisfying the recurrence
\begin{equation}\label[rec]{rec2}
W_n(z) = A(z)W_{n-1}(z)+B(z)W_{n-2}(z),
\end{equation}
where $A(z)$ and $B(z)$ are polynomials.
In this paper, we consider the polynomial sequence $\{W_n(z)\}_{n\ge0}$ 
satisfying the recurrence
\begin{equation}\label[rec]{rec2:linear}
W_n(z) = (az+b)W_{n-1}(z)+(cz^2+dz+e)W_{n-2}(z),
\end{equation}
with the polynomials $W_0(z)$ and $W_1(z)$ given. 
A number $z^*\in\mathbb{C}$ 
is said to be a {\em limit of zeros} of a sequence $\{W_n(z)\}_n$ of polynomials
if there is a zero $z_n$ of the polynomial $W_n(z)$ for each $n$ such that $\lim_{n\to\infty}z_n=z^*$. 
For any complex number $z=re^{i\theta}$ with $\theta\in(-\pi,\pi]$,
we use the square root notation~$\sqrt{z}$ to denote 
the number $\sqrt{r}e^{i\theta/2}$, 
which lies in the right half-plane $\theta\in(-\pi/2,\,\pi/2]$.
%The general formula in \cref{lem:00}
%is the base of our study.
%
%\begin{lemma}\label{lem:00}
%Let $A,B\in\mathbb{C}$.
%Suppose that $W_n=AW_{n-1}+BW_{n-2}$ for $n\ge 2$. Then 
%\[
%W_n=\begin{cases}
%{\displaystyle {\alpha_+\lambda_+^n+\alpha_-\lambda_-^n}},&\textrm{ if $\Delta\neq0$},\\[5pt]
%{\displaystyle  {\frac{W_0A+nh}{2}}\cdot\bigg({A\over 2}\bigg)^{n-1}},&\textrm{ if $\Delta=0$},
%\end{cases}
%\]
%for $n\ge 0$, 
%where $h=2W_1-W_0A$ and 
%\[
%\lambda_\pm=\frac{A\pm\sqrt{\Delta}}{2},\qquad
%\alpha_\pm=\frac{W_0\sqrt{\Delta}\pm h}{2\sqrt{\Delta}},
%\qquad\text{with $\Delta=A^2+4B$}.
%\] 
%\end{lemma}
%
We consider a sequence $\{W_n(z)\}_n$ of polynomials 
defined by \cref{rec2:linear}, where $a,b,c,d,e\in\mathbb{R}$ and $ac\ne0$.
Denote by $x_A=-b/a$ the zero of $A(z)$.
We employ the notation
\begin{align*}
\Delta(z)
&=A(z)^2+4B(z)
=(a^2+4c)z^2+(2ab+4d)z+(b^2+4e),\\
\lambda_\pm(z)&=\frac{A(z)\pm\sqrt{\Delta(z)}}{2},\quad\text{and}\\[4pt]
%h(z)&
%%=\frac{\partial g(z)}{\partial W_1(z)}
%=2W_1(z)-W_0(z)A(z),\\[4pt]
\alpha_\pm(z)&=
\frac{W_0(z)\sqrt{\Delta(z)}\pm (2W_1(z)-W_0(z)A(z))}{2\sqrt{\Delta(z)}}
\qquad\text{if $\Delta(z)\ne0$}.
\end{align*}

%We compute that 
%\begin{align*}
%B(x_A)&=\frac{a^2e-abd+b^2c}{a^2},\\
%B'(x_A)&=\frac{d B(z)}{dz}\Big|_{z=x_A}=\frac{ad-2bc}{a},\\
%\end{align*}

We denote the set of non-isolated limits
of zeros of the polynomials $W_n(z)$ by $\clubsuit$,
and denote the set of isolated limits of zeros by $\spadesuit$.
The clover symbol $\clubsuit$ is adopted
for the leaflets of a clover are not alone, while the spade symbol $\spadesuit$ 
appearing as a single leaflet represents isolation in comparison.
Every zero of the function $\Delta(z)$ is a non-isolated limit of zeros. 

\begin{lemma}[Beraha et al.]\label{lem:BKW}
Let $\{W_n(z)\}$
be the normalized polynomial sequence satisfying \cref{rec2:linear}.
Suppose that it satisfies the non-degeneracy conditions 
\begin{enumerate}[label=\emph{(N-\roman*)}]
\item\label{cond:rec2}
the sequence $\{W_n(z)\}$ does not satisfy a recurrence of order less than two, and
\item\label{cond:f<>0}
there is no $\omega\in\mathbb{C}$ with $\abs{\omega}=1$ 
such that $\lambda_+(z)\equiv\omega\lambda_-(z)$.
\end{enumerate}
Then $\clubsuit=\{z\in\mathbb{C}\colon \abs{\lambda_+(z)}=\abs{\lambda_-(z)}\}$ and 
$\spadesuit=\spadesuit_-\cup\spadesuit_+$,
where
\begin{align*}
\spadesuit_-
&=\{z\in\mathbb{C}\colon \Delta(z)\ne0,\ \alpha_-(z)=0,\ \abs{\lambda_-(z)}
>\abs{\lambda_+(z)}\},\quad\text{and}\quad\\
\spadesuit_+
&=\{z\in\mathbb{C}\colon \Delta(z)\ne0,\ \alpha_+(z)=0,\ \abs{\lambda_+(z)}
>\abs{\lambda_-(z)}\}.
\end{align*}
\end{lemma}

%By Beraha et al. \cite[Formula (6)]{BKW78}, \cref{cond:rec2} holds if and only if $\alpha_+(z)\not\equiv 0$ and $\alpha_-(z)\not\equiv 0$.

If Condition \ref{cond:rec2} is not satisfied, 
i.e., if the sequence satisfies a recurrence of order one,
then $W_n(z)=W_1(z)^n/W_0(z)^{n-1}$
for $n\ge 1$. Taking $n=2$ and by using 
\cref{rec2:linear} we obtain 
\[
W_1^2(z)-A(z)W_0(z)W_1(z)-B(z)W_0^2(z)=0.
\]
Condition \ref{cond:f<>0} is not satisfied
if and only if there is $\omega\in\mathbb{C}$ with $\abs{\omega}=1$
such that $\lambda_+(z)\equiv\omega\lambda_-(z)$. Then 
$A^2(z)=-(1+\omega)^2B(z)/\omega$.
It follows that $A(z)$ and $B(z)$ share a common zero. 
From now on we make the following assumptions:
\begin{enumerate}
\itemsep 5pt
\item\label[itm]{itm:Ni}
the polynomial $W_1^2(z)-A(z)W_0(z)W_1(z)-B(z)W_0^2(z)$ is not identically zero;
\item\label[itm]{itm:AB:share}
the polynomials $A(z)$ and $B(z)$ do not share a zero; and
\item\label[itm]{itm:W01:share}
the polynomials $W_0(z)$ and $W_1(z)$ do not share a zero.
\end{enumerate}
We say that a sequence $\{W_n(z)\}$ of polynomials
satisfying the above assumptions is \emph{general}. 
Note that \cref{itm:Ni,itm:AB:share} guarantee
the non-degeneracy Conditions \ref{cond:rec2} and \ref{cond:f<>0}.
We use the notation $\overline{z}$ for a number $z$
to denote the complex conjugate of $z$. 

\begin{theorem}\label{thm:spade}
Let $\{W_n(z)\}_{n\ge0}$ be a general sequence of polynomials 
satisfying \cref{rec2:linear}.
Then the set of isolated limits of their zeros is
\begin{equation}\label{il}
\spadesuit=
\brk[c]4{z\in\mathbb{C}\colon g(z)=0\text{ and }
\Re\brk3{\frac{W_1(z)\overline{A(z)}}{W_0(z)}}<\frac{\abs{A(z)}^2}{2}},
\end{equation}
where 
$g(z)=W_1^2(z)-A(z)W_0(z)W_1(z)-B(z)W_0^2(z)$.
\end{theorem}

\begin{proof}
Suppose that $W_0(z)\ne 0$. We notice that 
$\Re\brk1{W_1(z)\overline{A(z)}/W_0(z)}<\abs{A(z)}^2/2$
if and only if
\begin{equation}\label[ineq]{ineq:Re}
\Re\brk1{\overline{A(z)}h(z)/W_0(z)}<0,
\end{equation}
where $h(z)=2W_1(z)-W_0(z)A(z)$.
We shall show that $\spadesuit=S$, where 
\[
S=\brk[c]3{z\in\mathbb{C}\colon g(z)=0,\,
\Re\brk3{\frac{\overline{A(z)}h(z)}{W_0(z)}}<0}.
\]
Note that if $W_0(z)=0$, then $g(z)=0$ implies $W_1(z)=0$,
contradicting the assumption \cref{itm:W01:share}. Thus $W_0(z)\ne0$.

Let $z\in\spadesuit_-$. 
From the definition of $\alpha_-(z)$ and the expression of $g(z)$, 
we see that $g(z)=0$.
If $W_0(z)=0$, then $g(z)=0$ implies $W_1(z)=0$ 
and $z$ is a common zero of the polynomials $W_0$ and $W_1$, 
contradicting the premise. 
So $W_0(z)\ne0$.
By definition of $\spadesuit_-$, we have $\Delta(z)\ne0$ and $\alpha_-(z)=0$.
Thus $\sqrt{\Delta(z)}=h(z)/W_0(z)$.
From the definition of $\lambda_\pm(z)$, we have  
\begin{align*}
\abs{\lambda_-(z)}>\abs{\lambda_+(z)}
\iff
\Re\brk2{\overline{A(z)}\sqrt{\Delta(z)}}<0,
\end{align*}
which is equivalent to \cref{ineq:Re}. Thus $z\in S$ and $\spadesuit_-\subseteq S$.
Along the same lines, 
we can show that $\spadesuit_+\subseteq S$ and thus $\spadesuit\subseteq S$.

Conversely,
let $z\in S$. Note that 
\begin{equation}\label{g:d}
g(z)
=\begin{cases}
-\alpha_+(z)\alpha_-(z)\Delta(z),&\text{if $\Delta(z)\ne0$},\\[5pt]
h^2(z)/4,&\text{if $\Delta(z)=0$}.
\end{cases}
\end{equation}
From \cref{ineq:Re}, we know that $h(z)\ne0$.
Thus $\alpha_-(z)\alpha_+(z)=0$ by \cref{g:d}.
If $\alpha_-(z)=0$,
then \cref{ineq:Re} implies $\abs{\lambda_-(z)}>\abs{\lambda_+(z)}$, 
that is, $z\in\spadesuit_-$. 
For the other case $\alpha_+(z)=0$, we obtain $z\in\spadesuit_+$.
\end{proof}

\section{The geometry of root distribution}\label{sec:lim}
We employ the notation
\begin{align}
\Delta_B&=d^2-4ce,\notag\\
\Delta_\Delta&
=16(\Delta_B-a^2B(x_A)),\quad\text{and}\label{dd}\\
x_\Delta^\pm&=
\begin{cases}
\displaystyle
\frac{-ab-2d}{a^2+4c}\pm{\frac{\sqrt{\Delta_\Delta}}{2\abs{a^2+4c}}},
&\text{if $a^2+4c\ne0$,}\\[10pt]
\displaystyle
-\frac{b^2+4e}{2ab+4d},&\text{if $a^2+4c=0$ and $2ab+4d\ne0$.}
\end{cases}\notag
\end{align}
Here is our main result.
\begin{theorem}\label{thm:club}
Let $\{W_n(z)\}_{n\ge0}$ be a general sequence of polynomials 
satisfying \cref{rec2:linear}.
\begin{enumerate}
\item
If $\Delta_\Delta<0$, then
\[
\clubsuit=\begin{cases}
\wideparen{x_\Delta^-x_Ax_\Delta^+},
&\text{if $a^2+4c>0$},\\[5pt]
\wideparen{x_\Delta^-x_Cx_\Delta^+}\cup\mathbb{R},
&\text{if $a^2+4c<0$},
\end{cases}
\]
where $x_\Delta^\pm$ are the two zeros of $\Delta(z)$, and $x_C=x_A-2r$.
\item\label[itm]{itm:thm:club}
If $\Delta_\Delta\ge 0$, then
\begin{align*}
\clubsuit&=\begin{cases}
J_\Delta,&\text{if $B(x_A)<0$},\\[5pt]
J_\Delta\cup C_0,&\text{if $B(x_A)>0$},
\end{cases}\\
C_0&=\begin{cases}
\brk[c]1{z\in\mathbb{C}\colon\abs{z-x_A+r}=\abs{r}},&\text{if $B'(x_A)\ne0$},\\
\brk[c]{z\in\mathbb{C}\colon \Re z=x_A},&\text{if $B'(x_A)=0$},
\end{cases}
\end{align*}
where $J_\Delta=\{x\in\mathbb{R}\colon \Delta(x)\le0\}$ and $r=B(x_A)/B'(x_A)$.
\end{enumerate}
\end{theorem}

In order to prove \cref{thm:club}, we will need \cref{lem:club,lem:CS}.
\Cref{lem:ratio} is for establishing \cref{lem:club}.
\Cref{prop:S0} is to show \cref{lem:CS}.

\begin{lemma}\label{lem:ratio}
Let $u,v\in\mathbb{C}$. Then we have the following.
\begin{enumerate}
\item
If $\Re(v)\ne0$, then $u/v\in\mathbb{R}$ if and only if $u/v=\Re(u)/\Re(v)$.
\item
If $\Im(v)\ne0$, then $u/v\in\mathbb{R}$ if and only if $u/v=\Im(u)/\Im(v)$.
%\item
%If $t\in\mathbb{R}$, then $(u-tv)v<0$ if and only if $v\ne0$ and $u/v<t$.
\end{enumerate}
\end{lemma}
\begin{proof}
It is elementary if one considers the exponential 
form $Re^{i\theta}$ of complex numbers. 
We omit the proof details.
\end{proof}

We use the notation $\Cl(S)$ for a set $S\subseteq \mathbb{C}$ 
to denote the closure of $S$ in the complex plane $\mathbb{C}$ 
with respect to the standard Euclidean topology.

\begin{lemma}\label{lem:club}
For any general sequence $\{W_n(z)\}$ satisfying \cref{rec2:linear}, 
\begin{align*}
\clubsuit&=J_\Delta\cup \Cl(S_0\cap C_0),
\end{align*}
where 
\begin{align*}
S_0&=
\begin{cases}
\displaystyle
\brk[c]1{x+yi\colon aA(x)\brk1{aA(x)+2B'(x)}<0,\,x,y\in\mathbb{R}},
&\text{if $B'(x_A)\ne0$},\\
\displaystyle\brk[c]1{x+yi\colon c^2y^2(a^2+4c)+c\Delta_B<0,\,x,y\in\mathbb{R}},
&\text{if $B'(x_A)=0$}.
\end{cases}
\end{align*}
\end{lemma}
\begin{proof}
First of all, we claim that
$\clubsuit=\{z\in\mathbb{C}\colon A(z)\Delta(z)=0\text{ or }f(z)<-1\}$,
where $f(z)=4B(z)/A^2(z)$
for $z\in\mathbb{C}$ such that $A(z)\ne0$.
It is clear that any zero of~$A(z)$ or~$\Delta(z)$ belongs to the set $\clubsuit$.
Assume that $z\in\clubsuit$ and $A(z)\Delta(z)\ne0$.
Since $\Delta(z)=A^2(z)+4B(z)$,
we can infer the following equivalences:
\begin{align*}
&\abs{\lambda_+(z)}=\abs{\lambda_-(z)}\\
\iff&
\text{the vectors $A(z)$ and $\sqrt{\Delta(z)}$ are orthogonal}\\
\iff
&\text{the vectors $A^2(z)$ and $\Delta(z)$ have opposite directions}\\
\iff
&\text{$A^2(z)$ and $B(z)$ have opposite directions, 
and $\abs{A^2(z)}<\abs{4B(z)}$,}\\
\iff
&\text{the number $B(z)/A^2(z)$ is real and less than $-1/4$.}
\end{align*}
This proves the claim. 

Note that no point in the closed set $\clubsuit$ is isolated.
Since the zeros of the function $A(z)\Delta(z)$ are finite,
we deduce that $\clubsuit$ is the closure of the set 
\[
\clubsuit'=\{z\in\mathbb{C}\colon A(z)\Delta(z)\ne0,\,f(z)<-1\}.
\]
Let $z=x+yi\in\clubsuit'$, where $x,y\in\mathbb{R}$.
If $y=0$, then $A(z),\,B(z)\in\mathbb{R}$, 
and the inequality $f(z)<-1$ reduces to $\Delta(x)<0$. 
Therefore,
$\clubsuit'\cap\mathbb{R}=\{x\in\mathbb{R}\colon\Delta(x)<0\}$.

Below we suppose that $y\ne0$.
It is routine to compute that 
\begin{align}
A(z)^2&=A(x)^2-a^2y^2+2aA(x)yi,\label{A2z}\\
B(z)&=B(x)-cy^2+B'(x)yi,\label{Bz}\\
\Im f(z)
&=\frac{-4y\brk1{a^2B'(x_A)(x^2+y^2)+2(a^2e-b^2c)x+b(2ae-bd)}}{(A(x)^2+a^2y^2)^2}.
\label{Im:f}
\end{align}
Since $f(z)\in\mathbb{R}$, 
we will simplify the condition $\Im f(z)=0$ 
according to whether the coefficients of $x$ in the numerator of
the right hand side of \cref{Im:f} vanish.

Suppose that $B'(x_A)\ne0$. In this case, \cref{Im:f} can be recast as
\[
\Im f(z)=\frac{-4a^2B'(x_A)y}{(A(x)^2+a^2y^2)^2}\brk1{(x-x_A+r)^2+y^2-r^2}.
\]
It follows that $\Im f(z)=0$ if and only if $z\in C_0$.

We claim that $\Im\brk1{A(z)^2}\ne0$. 
In fact, if $\Im\brk1{A(z)^2}=0$, then $A(x)=0$ by \cref{A2z}, i.e., $x=x_A$. 
On the other hand, 
since $f(z)\in\mathbb{R}$, we find $4B(z)=A(z)^2f(z)\in\mathbb{R}$. 
By \cref{Bz}, we find $0=B'(x)=B'(x_A)$, a contradiction.
This proves the claim.
Now, by \cref{lem:ratio}, we can deduce that
\[
f(z)=\frac{4B(z)}{A(z)^2}
=\frac{4\Im B(z)}{\Im (A(z)^2)}
=\frac{4B'(x)y}{2aA(x)y}
=\frac{2B'(x)}{aA(x)}.
\]
By \cref{lem:ratio}, the inequality $f(z)<-1$ is equivalent to $z\in S_0$. 
In conclusion, 
\[
\clubsuit=\Cl(\clubsuit')
=\Cl\brk1{\{x\in\mathbb{R}\colon\Delta(x)<0\}\cup(C_0\cap S_0)}
=J_\Delta\cup\Cl(C_0\cap S_0).
\]

Now we suppose that $B'(x_A)=0$. If $a^2e-b^2c=0$, then 
\[
B(x_A)=\frac{b^2c-abd+a^2e}{a^2}
=\frac{2a^2e-abd}{a^2}=x_AB'(x_A)=0,
\]
contradicting the premise that $A(z)$ and $B(z)$ do not share 
a zero. Therefore $a^2e\ne b^2c$. By \cref{Im:f},
the condition $\Im f(z)=0$ implies
\[
x=-\frac{b(2ae-bd)}{2(a^2e-b^2c)}=x_A
\]
and vice versa. Thus $C_0$ is the line $\Re z=x_A$.
It follows that 
\[
\Re \brk1{A(z)^2}=A(x)^2-a^2y^2=-a^2y^2<0.
\]
Applying \cref{lem:ratio}, we obtain
\[
f(z)=\frac{4B(z)}{A(z)^2}
=\frac{4\Re B(z)}{\Re \brk1{A(z)^2}}
=\frac{4(B(x_A)-cy^2)}{-a^2y^2}
=\frac{4c^2y^2+\Delta_B}{a^2cy^2}.
\]
Therefore, by \cref{lem:ratio}, the inequality $f(z)<-1$ is equivalent to 
\[
z\in S_0'=\{x_A+yi\colon cy^2\brk1{cy^2(a^2+4c)+\Delta_B}<0,\,y\in\mathbb{R}\}.
\]
In conclusion
$
\clubsuit=\Cl(\clubsuit')
=\Cl\brk1{\{x\in\mathbb{R}\colon\Delta(x)<0\}\cup S_0'}
=J_\Delta\cup\Cl(C_0\cap S_0)$.
\end{proof}

We should mention that Tran \cite{Tra14} showed that all roots of the polynomials $W_n(z)$ 
defined by \cref{rec2} with initial polynomials $W_0(z)=1$ and $W_1(z)=B(z)$
densely lie in the set $\clubsuit$. In our framework, the polynomial $W_1(z)$ is not restricted 
to related with the coefficient polynomial $B(z)$ in \cref{rec2}.

In \cref{prop:S0}, we describe the geometric shape of the open set $S_0$, which will be used in the proof of \cref{lem:CS}. For convenience, we denote by $x_\Delta$ and $y_\Delta$ the real and imaginary part of the zeros of $\Delta(z)$, respectively, i.e., $x_\Delta^\pm=x_\Delta\pm i y_\Delta$, where
\begin{enumerate}
\item
if $a^2+4c\ne0$, then
\[
x_\Delta=-\frac{ab+2d}{a^2+4c}
\quad\text{and}\quad
y_\Delta=\frac{\sqrt{-\Delta_\Delta}}{2\abs{a^2+4c}};
\]
\item
if $a^2+4c=0$ and $2ab+4d\ne0$, then 
\[
x_\Delta=-\frac{b^2+4e}{2ab+4d}
\quad\text{and}\quad
y_\Delta=0.
\]
\end{enumerate}

\begin{proposition}\label{prop:S0}
If $B'(x_A)\ne0$, then the set $S_0$ has one of the following four shapes:
\begin{enumerate}
\itemsep=5pt
\item
if $a^2+4c>0$, 
then $S_0$ is a vertical strip with boundaries $\Re z=x_A$ and $\Re z=x_\Delta$;
\item
if $a^2+4c<0$,
then $S_0$ is the exterior of a vertical strip with boundaries $\Re z=x_A$ and $\Re z=x_\Delta$;
\item
if $a^2+4c=0$ and $B'(x_A)>0$, then
$S_0$ is the left half-plane with the boundary $\Re z=x_A$; 
\item
if $a^2+4c=0$ and $B'(x_A)<0$, then
$S_0$ is the right half-plane with the boundary $\Re z=x_A$.
\end{enumerate}
If $B'(x_A)=0$, then the set $S_0$ has one of the following four shapes:
\begin{enumerate}
\itemsep=5pt
\item
if $a^2+4c>0$ and $c\Delta_B<0$, 
then $S_0$ is the horizontal strip with boundaries $y=\pm y_\Delta$;
\item
if $a^2+4c<0$ and $c\Delta_B>0$,
then $S_0$ is the exterior of the horizontal strip with boundaries $y=\pm y_\Delta$;
\item
if both $a^2+4c$ and $c\Delta_B$ are nonnegative, then $S_0=\emptyset$;
\item
if the set $\{a^2+4c,\,c\Delta_B\}$ consists of a negative
and a nonpositive number,
then $S_0=\mathbb{C}\setminus\mathbb{R}$.
\end{enumerate}
\end{proposition}
\begin{proof}
Immediate from the definition of $S_0$.
\end{proof}
%
%\begin{lemma}\label{lem:xA}
%If $B'(x_A)\ne0$, then
%\[
%aA(x_C)\brk1{aA(x_C)+2B'(x_C)}
%=-\frac{a^2B(x_A)\Delta_\Delta}{4B'(x_A)^2}.
%\]
%\end{lemma}

Note that $\Delta_\Delta=16(d+ab/2)^2\ge0$ if $a^2+4c=0$.
Now we are in a position to present the shape of the closure of $C_0\cap S_0$.

\begin{lemma}\label{lem:CS}
We have
\begin{align*}
\Cl(C_0\cap S_0)&=\begin{cases}
\wideparen{x_\Delta^-x_Ax_\Delta^+},
&\text{if $\Delta_\Delta<0$ and $a^2+4c>0$};\\[5pt]
\wideparen{x_\Delta^-x_Cx_\Delta^+},
&\text{if $\Delta_\Delta<0$ and $a^2+4c<0$};\\[5pt]
C_0,
&\text{if $\Delta_\Delta\ge0$ and $B(x_A)>0$};\\[5pt]
\emptyset,&\text{if $\Delta_\Delta\ge0$ and $B(x_A)<0$}.
\end{cases}
\end{align*}
where $x_C=x_A-2r$
is the real point on the circle $C_0$ that is not $x_A$. 
In particular, when $\Delta_\Delta<0$ and $B'(x_A)=0$, 
the arc $\wideparen{x_\Delta^-x_Ax_\Delta^+}$ 
degenerates to the line segment with ends~$x_\Delta^\pm$, 
and the arc $\wideparen{x_\Delta^-x_Cx_\Delta^+}$ 
degenerates to the union $\{x_A+yi\colon\abs{y}\ge y_\Delta\}$ of two rays.
\end{lemma}

\begin{proof}
We need to deal with $4$ cases according to the sign of $\Delta_\Delta$ and the sign of $B'(x_A)$.

Suppose that $\Delta_\Delta<0$ and $B'(x_A)\ne0$. Then the set
\begin{equation}\label{eq:C0}
C_0=\brk[c]1{x+yi\colon(x-x_A+r)^2+y^2=r^2,\,x,y\in\mathbb{R}}
\end{equation}
is a circle, $C_0\cap \mathbb{R}=\{x_A,\,x_C\}$,
and $x_\Delta^\pm\in C_0$. As a result, $C_0$ is the unique circle passing through the points $x_A$, $x_\Delta^-$ and $x_\Delta^+$. On the other hand,
by \cref{prop:S0}, we see that 
\[
S_0=\brk[c]1{x+yi\colon aA(x)\brk1{aA(x)+2B'(x)}<0,\,x,y\in\mathbb{R}}
\]
is either the vertical strip with boundaries $\Re z=x_A$ and $\Re z=x_\Delta$ (when $a^2+4c>0$), 
or the exterior of this strip (when $a^2+4c<0$).
This proves the first two cases that the closure
is the arc $\wideparen{x_\Delta^-x_Ax_\Delta^+}$ or
the arc $\wideparen{x_\Delta^-x_Cx_\Delta^+}$ correspondingly.

Suppose that $\Delta_\Delta<0$ and $B'(x_A)=0$. 
From definition, the set $C_0$ is the vertical line $\Re z=x_A$;
by \cref{prop:S0}, the set 
\[
S_0=\brk[c]1{x+yi\colon cy^2\brk1{cy^2(a^2+4c)+\Delta_B}<0,\,x,y\in\mathbb{R}}
\]
is the horizontal strip with boundaries $y=\pm y_\Delta$ (when $a^2+4c>0$), 
or the complement of the horizontal strip (when $a^2+4c<0$).
Therefore, the intersection closure $\Cl(C_0\cap S_0)$
is the closed vertical line segment with ends $(x_A,\pm y_\Delta)$ if $a^2+4c>0$, and it is the union of two rays obtained from the line $\Re z=x_A$ by removing the open line segment with those ends if $a^2+4c<0$.

Suppose that $\Delta_\Delta\ge0$ and $B'(x_A)\ne0$.
If $a^2+4c\ne0$, then we can compute that 
\begin{align}
x_A-x_\Delta
&=\frac{2B'(x_A)}{a^2+4c},\notag\\
(x_A-x_C)(x_A-x_\Delta)
&=\frac{4B(x_A)}{a^2+4c},\label{A-C:A-d}\\
(x_\Delta-x_A+r)^2-r^2
&=\frac{\Delta_\Delta}{4(a^2+4c)^2}\ge0.\notag
\end{align}
In view of \cref{eq:C0}, 
we find that the line $\Re z=x_\Delta$ 
is either tangent to the circle $C_0$, or disjoint from $C_0$.
Now, by checking each of the four cases in \cref{prop:S0} that the set $S_0$
is a strip or the exterior of a strip, we obtain that the closure 
$\Cl(C_0\cap S_0)$ equals either $C_0$ or $\emptyset$. 

If $a^2+4c>0$, then the set $S_0$ is the strip with boundaries $\Re z=x_A$ and $\Re z=x_\Delta$. In this case, the closure is $C_0$ if and only if 
the points $x_C$ and $x_\Delta$ lie on the same side of the line $\Re z=x_A$, which is equivalent to $B(x_A)>0$ by \cref{A-C:A-d}.  If $a^2+4c<0$, then $S_0$ is the exterior of the strip. In this case, the closure is $C_0$ if and only if $x_C$ and $x_\Delta$ lie on different sides of $x_A$, which is also equivalent to $B(x_A)>0$ by \cref{A-C:A-d}. 
Since the polynomials $A(z)$ and $B(z)$ do not share a zero,
$B(x_A)\ne0$. This proves the desired results for the closure
to be $C_0$ or $\emptyset$ as if $a^2+4c\ne 0$.

Otherwise $a^2+4c=0$. In this case, the set $S_0$ is a half plane with the boundary $\Re z=x_A$, which is tangent to the circle $C_0$. It follows that the closure is either $C_0$ or $\emptyset$. We shall determine the closure by considering whether $x_C\in S_0$.
If $B'(x_A)>0$, then~$S_0$ is the left half plane. In this case,
$x_C\in S_0$ if and only if $x_C<x_A$, that is, $r>0$, or equivalently $B(x_A)>0$.
If $B'(x_A)<0$, then $S_0$ is the right half plane. In this case, $x_C\in S_0$ if and only if $x_C>x_A$, that is, $r<0$, or equivalently $B(x_A)>0$ again. To sum up, we find that the closure is $C_0$ if and only if $B(x_A)>0$. It follows that the closure is empty if and only if $B(x_A)<0$.

At last,
suppose that $\Delta_\Delta\ge 0$ and $B'(x_A)=0$. Now, the set $C_0$ is the line $\Re z=x_A$, and the set $S_0$ is either empty (when both $a^2+4c$ and $c\Delta_B$ are nonnegative) or equal to~$\mathbb{C}\backslash\mathbb{R}$ (when $\{a^2+4c,\,c\Delta_B\}$ consists of a  negative number and a nonpositive number). Thus the closure $\Cl(C_0\cap S_0)$ equals $\emptyset$ or $C_0$ correspondingly.
It suffices to show that these two cases correspond to $B(x_A)<0$ and $B(x_A)>0$ respectively. In fact,
when $c\Delta_B\ge 0$, we can infer that $B(x_A)$ as the  extremal value of the quadratic real function $B(x)$ is nonpositive. Moreover, the value $B(x_A)$ has to be negative since $A(z)$ and $B(z)$ do not share a zero. When $c\Delta_B\le 0$ and $a^2+4c\le0$, we have $c<0$ and $\Delta_B\ge 0$,
which implies $B(x_A)>0$ in the same fashion.

This completes the proof.
\end{proof}

Now \cref{thm:club} follows from \cref{lem:club,lem:CS} immediately. 

We do not consider the case that $\Delta(z)=A^2(z)+4B(z)$ is identically zero, since that would imply common zeros of the polynomials $A(z)$ and $B(z)$.
%
%\begin{proposition}\label{prop:J}
%Let $\{W_n(z)\}_{n\ge0}$ be a general sequence of polynomials satisfying \cref{rec2:linear}.
%Suppose that $\Delta_\Delta\ge 0$. Then 
%\begin{equation}\label{JDelta:Delta>=0}
%J_\Delta=\begin{cases}
%[x_\Delta^-,\,x_\Delta^+] &\text{if $a^2+4c>0$};\\
%(-\infty,\,x_\Delta^-]\cup[x_\Delta^+,\,+\infty),
%&\text{if $a^2+4c<0$};\\
%[x_\Delta,\,+\infty),&\text{if $a^2+4c=0$ and $ab+2d<0$};\\
%(-\infty,\,x_\Delta],&\text{if $a^2+4c=0$ and $ab+2d>0$};\\
%\mathbb{R},&\text{if $a^2+4c=ab+2d=0$ and $b^2+4e<0$};\\
%\emptyset,&\text{if $a^2+4c=ab+2d=0$ and $b^2+4e>0$}.
%\end{cases}
%\end{equation}
%Moreover, if $B(x_A)>0$, then
%\begin{equation}\label{J:cap:C0}
%J_\Delta\cap C_0=\begin{cases}
%\{x_C\},&\text{if $B'(x_A)\ne0$},\\
%\emptyset,&\text{if $B'(x_A)=0$}.
%\end{cases}
%\end{equation}
%\end{proposition}
%\begin{proof}
%\Cref{JDelta:Delta>=0} follows from definition directly.
%Note that $\Delta(x_A)=4B(x_A)$.
%If $B(x_A)>0$, then $\Delta(x_A)>0$ and $x_A\not\in J_\Delta$. 
%This proves $J_\Delta\cap C_0=\emptyset$ if $B'(x_A)=0$.
%When $B'(x_A)\ne0$, 
%the set $C_0$ is a circle intersecting $\mathbb{R}$ at $x_A$ and $x_C$. Since
%\[
%\Delta(x_C)=-\frac{B(x_A)\Delta_\Delta}{4B'(x_A)^2}\le0,
%\]
%we obtain $J_\Delta\cap C_0=\{x_C\}$.
%\end{proof}

\begin{theorem}\label{thm:shape:12}
Let $\{W_n(z)\}_{n\ge0}$ be a general sequence of polynomials satisfying \cref{rec2:linear}.
Then the geometric shape of the set of non-isolated limits of zeros of the polynomials $W_n(z)$ is one of the following:
\begin{enumerate}
\item\label[itm]{itm:arc}
an arc with non-real ends $x_\Delta^-$ and $x_\Delta^+$, whose degenerate form is the vertical line segment with ends $x_\Delta^\pm$; 
\item\label[itm]{itm:arc:complement}
the union of $\mathbb{R}$ and an arc with non-real ends $x_\Delta^-$ and $x_\Delta^+$; the degenerate form of the arc is the union of two disjoint vertical rays with the ends $x_\Delta^\pm$, respectively;
\item\label[itm]{itm:J}
the real line $\mathbb{R}$, or 
a bounded interval of positive length, or a ray, or the union of two disjoint rays on~$\mathbb{R}$; 
\item\label[itm]{itm:J:line}
the union of a vertical line and a member in \cref{itm:J},
with empty intersection;
\item\label[itm]{itm:J:circle}
the union of a circle and a member in \cref{itm:J},
with a singleton intersection.
\item\label[itm]{itm:line:circle}
a circle, whose degenerate form is a vertical line.
\end{enumerate}
\end{theorem}
\begin{proof}
\Cref{itm:arc,itm:arc:complement} correspond to Case $\Delta_\Delta<0$ in \cref{thm:club}. \Cref{itm:J,itm:J:line,itm:J:circle,itm:line:circle} correspond to Case $\Delta_\Delta\ge0$ in \cref{thm:club}.
Note that the set of limits of zeros of all $W_n(z)$ is closed considering the Euclidean topology in $\mathbb{C}$, and that the set $\clubsuit$ is not a singleton since it is non-isolated.
We think it suffices to discuss some extremal cases.

Since $\clubsuit$ contains no isolated points, when $J_\Delta=[x_\Delta^-,\,x_\Delta^+]$, it does not degenerate to a singleton as $x_\Delta^-=x_\Delta^+$ except when it is contained in $C_0$. By the definition, 
$J_\Delta$ degenerates to a singleton only when $x_\Delta^-=x_\Delta^+=x_C$, and in that case $\clubsuit$ is the circle~$C_0$.

When $a^2+4c=ab+2d=0$, the polynomial $\Delta(z)=b^2+4e$ is a constant. In this case, by \cref{dd}, we find 
\[
a^2B(x_A)=\Delta_B=d^2-4ac=a^2(b^2+4e)/4.
\]
If $b^2+4e<0$, then $B(x_A)<0$.
By \cref{thm:club} and the definition of $J_\Delta$, $\clubsuit=J_\Delta=\mathbb{R}$.
If $b^2+4e>0$, then $B(x_A)>0$ and $\clubsuit=C_0$. Moreover,
\[
B'(x_A)=2cx_A+d=2c\brk2{-\frac{b}{a}}-\frac{ab}{2}
=-\frac{b(a^2+4c)}{8a}=0.
\]
Thus $\clubsuit$ is the vertical line $\Re z=x_A$.
This completes the proof.
\end{proof}

A polynomial is said to be \emph{real-rooted}
(resp., \emph{Hurwitz}, \emph{Schur stable}) 
if all its zeros lie in the real axis
(resp., the closed left half-plane, the closed unit disk).
We say that the polynomials in a sequence $\{W_n(z)\}_{n\ge0}$ is \emph{eventually real-rooted} 
(resp., \emph{eventually Hurwitz}, \emph{eventually Schur stable})
if there exists $N>0$ such that $W_n(z)$ is real-rooted 
(resp., Hurwitz, Schur stable) 
for all $n>N$.

\begin{corollary}
Let $\{W_n(z)\}$ be a general sequence of polynomials satisfying \cref{rec2:linear}. Let $\Delta(z)=A^2(z)+4B(z)$ be the discriminant polynomial of the recurrence. Then we have the following.
\begin{enumerate}
\item\label[itm]{itm:realrooted}
If the polynomials $W_n(z)$ are eventually real-rooted, 
then $\Delta(x_A)<0$. Moreover, $\Delta(z)$ is either a negative constant or real-rooted.
\item\label[itm]{itm:Hurwitz}
If the polynomials $W_n(z)$ are eventually Hurwitz, then $\Delta(z)$ has a positive leading coefficient, and the polynomial $A(z)\Delta(z)$ is Hurwitz.
\item\label[itm]{itm:SchurStable}
If the polynomials $W_n(z)$ are eventually Schur stable, 
then $\Delta(z)$ is quadratic with a positive leading coefficient, and
the polynomial $A(z)\Delta(z)$ is Schur stable.
\end{enumerate}
\end{corollary}
\begin{proof}
\noindent\cref{itm:realrooted}.
Suppose that the polynomials $W_n(z)$ are eventually real-rooted. Then $\clubsuit\subseteq\mathbb{R}$.
By \cref{thm:club}, this happens either when $\Delta_\Delta\ge0$ and $B(x_A)<0$, or when~$\Delta(z)$ is a negative constant.
In the former case, $\Delta(x_A)=4B(x_A)<0$; in the latter $\Delta(x_A)$ is the constant which is negative.

\noindent\cref{itm:Hurwitz}.
Suppose that the polynomials $W_n(z)$ are eventually Hurwitz.
Then $\clubsuit$ has no intersection with the open right half-plane.
By \cref{thm:club} and the definition of $J_\Delta$, 
either $x_A$ or $x_\Delta^+$ is a right most point in $\clubsuit$. 
This proves that $A(z)\Delta(z)$ is Hurwitz.
If the leading coefficient of $\Delta(z)$ is negative,
then $\clubsuit$ contains large real numbers,
which contradicts the Hurwitz condition.

\noindent\cref{itm:SchurStable}.
Suppose that the polynomials $W_n(z)$ are eventually Schur stable.
Then $\clubsuit\subseteq D$,
where $D$ is the closed unit disk.
By \cref{thm:club} and the definition of $J_\Delta$, 
if $a^2+4c\le0$, then the set $\clubsuit$ is unbounded, 
contradicting the Schur stable condition.
Thus $\Delta(z)$ is quadratic with a positive leading coefficient. 

If $\Delta_\Delta<0$, 
then $\clubsuit=\wideparen{x_\Delta^-x_Ax_\Delta^+}$ 
is an arc which is symmetric about the real axis.
Therefore, $\clubsuit\subseteq D$ is equivalent to say that $x_A\in D$
and $x_\Delta^\pm\in D$. In other words, $A(z)\Delta(z)$ is Schur stable.

Now we can suppose that $\Delta_\Delta\ge0$. 
If $B(x_A)<0$, 
then $\clubsuit=J_\Delta=[x_\Delta^-,\,x_\Delta^+]$ is an interval, 
and $\clubsuit\subseteq D$ is equivalent to $x_\Delta^\pm\in D$. 
It follows that $x_A\in J_\Delta\subseteq D$. Hence $A(z)\Delta(z)$ is Schur stable.
If $B(x_A)>0$, then $B'(x_A)\ne0$ and $\clubsuit=J_\Delta\cup C_0$,
where $C_0$ is the non-degenerate circle with a diameter from $x_A$ to $x_C$,
and $J_\Delta=[x_\Delta^-,\,x_\Delta^+]$ is an interval.
By the value of $\Delta(x_C)$, we have $x_C\in J_\Delta$.
Therefore, $\clubsuit\subseteq D$ is equivalent to say that,
again, $x_A\in D$ and $x_\Delta^\pm\in D$. 
\end{proof}

Before ending this paper,
we provide two examples satisfying \cref{rec2:linear}
which comes from topological graph theory; see also~\cite{WZ18X}.

\begin{example}\label{eg:4:2:16:0:0}\label{ex:ladder}
The ladder graph $L_n$ is the cartesian product of the complete graph~$K_2$
with the path $P_{n+2}$ on $n+2$ vertices. 
Chen and Gross~\cite[Theorem 3.1]{CG18} showed that
the Euler-genus polynomials $W_n(z)$ of $L_n$ satisfies the recurrence
\begin{equation}\label[rec]{rec:Euler-genus:ladder}
W_n(z)=(2+4z)W_{n-1}(z)+16z^2W_{n-2}(z),
\end{equation}
with $W_0(z)=1+z$ and $W_1(z)=2+6z+8z^2$.
It is direct to compute the discriminant
\[
\Delta(z)=4(1+4z+20z^2).
\]
By \cref{thm:spade,thm:club}, we obtain
$\spadesuit=\emptyset$
and $\clubsuit=\wideparen{x_\Delta^-x_Ax_\Delta^+}$ respectively, 
where $x_\Delta^\pm=-0.1\pm0.2i$ and $x_A=-0.5$.
The root distribution is illustrated in \cref{fig:3:-16:5:20}. 
%Using mathematical software we find that every polynomial $W_n(z)$ for $n\le 100,000$ is log-concave.
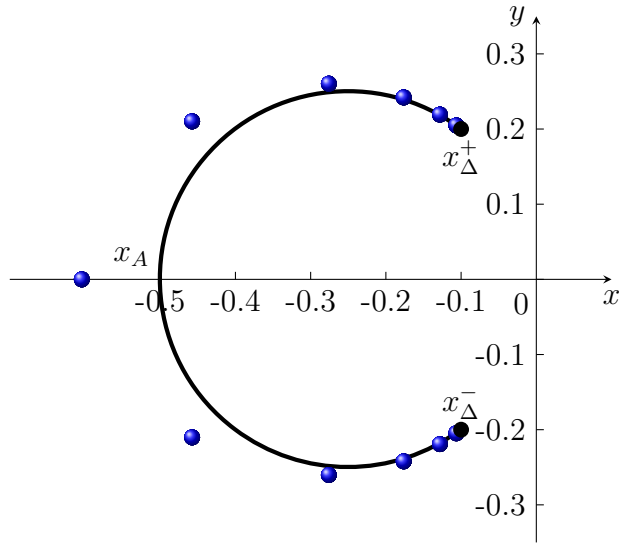
\begin{figure}[h]
\centering
\begin{tikzpicture}
\draw[->] (-7,0)--(1,0) node[below] {$x$};
\draw[->] (0,-3.5)--(0,3.5) node[left] {$y$};
\foreach \x in {0,1,...,6}
{\draw[yshift=\x cm] (0,-3) -- (0.1,-3);};
\foreach \x in {-5,-4,...,-1}
{\draw[xshift=\x cm] (0,0) -- (0,0.1);};
\node[below] at (-0.2,-0.05){0};
\node[below] at(-5,0){-0.5};
\node[below] at(-4,0){-0.4};
\node[below] at(-3,0){-0.3};
\node[below] at(-2,0){-0.2};
\node[below] at(-1,0){-0.1};
\node[left] at(0,-1){-0.1};
\node[left] at(0,-2){-0.2};
\node[left] at(0,-3){-0.3};
\foreach \y in {1,2,3} \node[left] at(0,\y){0.\y};
\draw[club] (-1,2) arc (53:307:2.5 cm);
\shade[ball color=blue] (-4.5756,2.1021) circle (3pt);
\shade[ball color=blue] (-4.5756,-2.1021) circle (3pt);
\shade[ball color=blue] (-2.7579,2.6028) circle (3pt);
\shade[ball color=blue] (-2.7579,-2.6028) circle (3pt);
\shade[ball color=blue] (-1.7627, 2.4209) circle (3pt);
\shade[ball color=blue] (-1.7627, -2.4209) circle (3pt);
\shade[ball color=blue] (-1.2815, 2.1922) circle (3pt);
\shade[ball color=blue] (-1.2815, -2.1922) circle (3pt);
\shade[ball color=blue] (-1.0632, 2.0477) circle (3pt);
\shade[ball color=blue] (-1.0632, -2.0477) circle (3pt);
\shade[ball color=blue] (-6.042, 0) circle (3pt);
\fill (-1,-2) circle (3pt);
\node[above] at (-1,-2){$x_\Delta^-$};
\fill (-1,2) circle (3pt);'
\node[below] at (-1,2){$x_\Delta^+$};
\node[left] at (-5,.3){$x_A$};
\end{tikzpicture}
\caption{
The zeros of the polynomial $W_{6}(z)$ in \cref{eg:4:2:16:0:0} are illustrated by balls. The ultra thick arc is the set $\clubsuit$ of non-isolated limits of zeros of the polynomials $W_n(z)$. }
\label{fig:3:-16:5:20}
\end{figure}
\end{example}

\begin{example}\label{ex:P4}
The Euler-genus polynomials
of certain $P_4$-linear graph-with-spiders
can be obtained by replacing the initial polynomials in \cref{eg:4:2:16:0:0} by
\[
W_0(z)=144(1+z)^4\quad\text{and}\quad
W_1(z)=64(1+z)^2(4+16z+31z^2+21z^3);
\]
see \cite[Subsection 4.5]{CG18}. 
It follows that $W_n(-1)=0$ for each $n$.
Consider the sequence $\mathcal{J}=\{J_n(z)\}_{n\ge0}$ of polynomials defined by
$J_n(z)=W_n(z)/(1+z)^2$.
Then the sequence~$\mathcal{J}$
satisfies \cref{rec:Euler-genus:ladder}.
The set of non-isolated limits of zeros of the polynomials in~$\mathcal{J}$ 
is the arc same to the one in \cref{eg:4:2:16:0:0}.
By \cref{thm:spade}, the set of isolated limits of zeros 
of the polynomials in $\mathcal{J}$ consists of a real zero of the polynomial 
$4+28z+169z^2+312z^3+171z^4$,
which equals $-0.8102$ approximately.
\end{example}

More examples can be found from 
Chen, Gross, Mansour and Tucker~\cite{CGMT16}.

\end{document}